\documentclass[12pt]{amsart}
\usepackage{amsfonts}
\usepackage{latexsym}
\usepackage{amssymb}
\usepackage{amsmath}
\usepackage{color}
\usepackage{bbm}
\usepackage{tikz}
\usepackage{enumerate}
\usepackage{mathrsfs}
\usepackage{todonotes}

\usepackage[left=2.5cm, top=2.5cm,bottom=2.5cm,right=2.5cm]{geometry}

\newcommand{\R}{\mathbb R}
\newcommand{\N}{\mathbb N}

\newcommand{\E}{\mathbb E}
\newcommand{\rate}{\mathbb I}
\newcommand{\Pro}{\mathbb P}

\newcommand{\vol}{\mathrm{vol}}

\def\dint{\textup{d}}

\newcommand{\B}{\ensuremath{{\mathbb B}}}

\newcommand{\eps}{\varepsilon}

\DeclareMathOperator{\id}{id}

\newtheorem{thm}{Theorem}[section]
\newtheorem{cor}[thm]{Corollary}
\newtheorem{lemma}[thm]{Lemma}
\newtheorem{df}[thm]{Definition}
\newtheorem{proposition}[thm]{Proposition}

{
\theoremstyle{definition}

\newtheorem{rmk}[thm]{Remark}
}

\newcommand{\toas}{\overset{{\rm a.s.}}{\underset{n\to\infty}\longrightarrow}}
\newcommand{\ton}{\overset{}{\underset{n\to\infty}\longrightarrow}}

\usepackage{nomencl}
\makenomenclature

\begin{document}


\title[]{A new look at random projections of the cube \\and general product measures}

\author[Z. Kabluchko]{Zakhar Kabluchko}
\address{Zakhar Kabluchko: Institut f\"ur Mathematische Stochastik, Westf\"alische Wilhelms-Uni\-ver\-sit\"at M\"unster, Germany.}
\email{zakhar.kabluchko@uni-muenster.de}

\author[J. Prochno]{Joscha Prochno}
\address{Joscha Prochno: Institut f\"ur Mathematik \& Wissenschaftliches Rechnen, Karl-Franzens-Universit\"at Graz, Austria} \email{joscha.prochno@uni-graz.at}

\author[C. Th\"ale]{Christoph Th\"ale}
\address{Christoph Th\"ale: Faculty of Mathematics, Ruhr University Bochum, Germany.} \email{christoph.thaele@rub.de}

\keywords{Cube, Gaussian random matrices, Gaussian projections, high-dimensional probability, Hausdorff distance, large deviations, law of large numbers, random projections, Stiefel manifold}
\subjclass[2010]{Primary: 60F10 Secondary: 52A22, 52A23, 60B20, 60G57}



\begin{abstract}
A strong law of large numbers for $d$-dimensional random projections of the $n$-dimensional cube is derived. It shows that with respect to the Hausdorff distance a properly normalized random projection of $[-1,1]^n$ onto $\R^d$ almost surely converges to a centered $d$-dimensional Euclidean ball of radius $\sqrt{2/\pi}$, as $n\to\infty$. For every point inside this ball we determine the asymptotic number of vertices and the volume of the part of the cube projected `close' to this point. Moreover, large deviations for random projections of general product measures are studied. Let $\nu^{\otimes n}$ be the $n$-fold product measure of a Borel probability measure $\nu$ on $\R$, and let $I$ be uniformly distributed on the Stiefel manifold of orthogonal $d$-frames in $\mathbb{R}^n$. It is shown that the sequence of random measures $\nu^{\otimes n}\circ(n^{-1/2}I^*)^{-1}$, $n\in\N$, satisfies a large deviations principle with probability $1$. The rate function is explicitly identified in terms of the moment generating function of $\nu$. At the heart of the proofs lies a transition trick which allows to replace the uniform projection by the Gaussian one. A number of concrete examples are discussed as well, including the uniform distributions on the cube $[-1,1]^n$ and the discrete cube $\{-1,1\}^n$ as a special cases.
\end{abstract}

\maketitle


\section{Introduction \& Main Results}

\subsection{General introduction}

One of the central aspects of the theory of high-dimensional probability 
is the study of convex bodies, i.e., compact and convex sets with non-empty interior, as the dimension of the ambient space $\R^n$ tends to infinity. The investigation of their geometry and asymptotic shape combines methods and ideas from probability theory, classical convex geometry, and functional analysis. The understanding of such high-dimensional structures, in particular through their lower-dimensional projections, has turned out to be crucial in numerous applications, among others in statistics and machine learning in the form of dimensionality reduction, clustering, or regression \cite{BM2001, FB2003, MM2010}, in compressed sensing when studying general performance bounds for sparse recovery methods or low-rank matrix recovery \cite{CDK2015, FPRU2010, HPV2017}, or in information-based complexity when investigating the tractability of multivariate integration or approximation problems \cite{HKNPU2019_survey, HKNPU2019, HPU2019}.

It has often been the probabilistic point of view that has lead to groundbreaking new results and insights into high-dimensional structures. In the classical theory as well as in modern developments, laws of large numbers and central limit theorems -- describing the typical behavior of random objects -- have been obtained for various quantities related to the geometry of convex bodies (see, e.g., \cite{APT2017, KPT17CCM, KPT19Part2, MeckesMeckes, PPZ2014, ProchnoThaeleTurchiChapter, SchechtmanSchmuckenschlaeger, SchmuckenschlaegerCLT} or the textbooks \cite{AsymptoticGeometricAnalysisBookPart1,IsotropicConvexBodies}) and provided us with a quite deep understanding of the asymptotic behavior of convex bodies. However, the beauty of universality described by central limit phenomena comes at a price. For instance, the celebrated central limit theorem for convex bodies of Klartag \cite{KlartagCLT,KlartagCLT2} roughly speaking says that most lower-dimensional marginals of a convex body in high dimensions are close to being Gaussian. Against this light, it is not possible to distinguish between different initial bodies via their lower-dimensional marginals on a scale of Gaussian fluctuations.

On the other hand, it is well known in probability theory that, in sharp contrast to the central limit theorem, the large deviations behaviour of a sum of random variables is very sensitive to the distribution of the terms in the sum. Very recently, this point of view was transferred to high-dimensional probability and asymptotic convex geometry by Gantert, Kim, and Ramanan \cite{GKR} who studied the large deviations behaviour of one-dimensional random projections of $\ell_p$-balls in $\R^n$, as $n\to\infty$. In a quite short period of time, their work triggered a number of papers. Among others, the so-called annealed case in \cite{GKR} was generalized to higher-dimensional marginals in \cite{APT2018}. The large deviations behavior of the $q$-norm of random points in high-dimensional $\ell_p$-balls was studied in \cite{KPT17CCM, KPT19Part2}. In \cite{KimRamanan}, Sanov-type large deviations for high-dimensional $\ell_{p}$-spheres were obtained and a non-commutative version for the empirical spectral measure of random matrices in Schatten unit balls has recently been proved in \cite{KPT2019}. The present paper continues and complements this body of current research as we shall explain now.

We start by investigating the shape of $d$-dimensional random projections of the $n$-dimensional cube, where $d$ is fixed and $n\to\infty$. We shall show that, after a suitable rescaling, the sequence of randomly projected cubes satisfies a strong law of large numbers, see Theorem \ref{thm:SLLN}. Roughly speaking, this result says that, typically, such random projections are close to a Euclidean ball of radius $\sqrt{2n/\pi}$. We then move on with a description of how the projected mass of the cube or the projections of its vertices are distributed inside this ball. More generally, we shall provide an explicit description of the large deviations of an $n$-fold product measure under $d$-dimensional random projections, where $d$ is fixed and $n\to\infty$, see Theorem \ref{thm:as LDP}.

\subsection{Main results}

In this section we present the main results of this paper, which are stated as Theorem \ref{thm:SLLN} and Theorem \ref{thm:as LDP}. The proof of Theorem \ref{thm:SLLN} as well as the proof of Theorem \ref{thm:as LDP} is build upon an analysis of Gaussian random projections and at its heart relies on a transition trick, which allows us to go from the Gaussian setting to the general one of Stiefel manifolds.

We start by introducing the set-up we shall be working in. Let $(g_{ij})_{i,j\in\N}$ be an infinite array of independent standard Gaussian random variables.  Given some $n,d\in\N$ with $d\leq n$ consider a Gaussian random $(d\times n)$-matrix  $G=(g_{ij})_{i,j=1}^{d,n}:\R^n\to\R^d$ and a pair of adjoint  linear mappings $I:\R^d\to\R^n$ and $I^*:\R^n \to \R^d$ given by
\begin{align}\label{eq:defIStern}
I = G^*(GG^*)^{-1/2} \qquad\text{and}\qquad  I^* = (GG^*)^{-1/2}G.
\end{align}
 As we shall explain in Lemma \ref{lem:uniform distribution Stiefel manifold} below, this particular choice of $I$ and $I^*$ means that  the collection of  columns of $I$ is uniformly distributed on the Stiefel manifold $\mathbb V_{n,d}$ of orthonormal $d$-frames in $\R^n$. Therefore, $I: \R^d\to \R^n$ is a random isometric embedding whose image $I\R^d$ is a $d$-dimensional linear subspace of $\R^n$ distributed according to the Haar probability measure on the Grassmannian of all such subspaces. Hence, we can regard the operator $I I^*:\R^n \to \R^n$ as an orthogonal projection onto the random, uniformly distributed $d$-dimensional linear subspace $I\R^d$. We shall be interested in randomly projected objects such as the randomly projected cube $II^*([-1,1]^n)$. Since the operator $I$ is an isometry between $\R^d$ and its image $I \R^d$, we can use it to identify $II^*([-1,1]^n)\subset \R^n$ and $I^*([-1,1]^n)\subset \R^d$. By abuse of language, we shall refer to the latter object as the randomly projected cube.

The first object we study is the `typical' shape of the (properly normalized) randomly projected cube $I^*([-1,1]^n)$, as $n\to\infty$. To compare its shape to another one, we use the classical notion of Hausdorff distance. More precisely, we denote by $\mathcal K(\R^d)$ the metric space of compact subsets of $\R^d$ endowed with the Hausdorff distance $\dint_H$, which, for $A,B\in \mathcal K(\R^d)$, is given by
\[
\dint_H\big(A,B\big):= \max\Big\{ \inf\big\{r\geq 0\,:\, A\subseteq B + \B_2^d(0,r) \big\}\,,\, \inf\big\{r\geq 0\,:\, B\subseteq A + \B_2^d(0,r) \big\} \Big\},
\]
where $\B_2^d(0,r)$ denotes the Euclidean ball centered at the origin and having radius $r>0$. The next result is a strong law of large numbers and shows that typically, after a suitable rescaling, $I^*([-1,1]^n)$ is close to the centered Euclidean ball of radius $\sqrt{2/\pi}$.

\begin{thm}\label{thm:SLLN}
For fixed $d\in\N$ we have that
$$
\dint_H\bigg({1\over\sqrt{n}}I^*([-1,1]^n), \B_2^d\Big(0,\sqrt{\tfrac{2}{\pi}}\Big)\bigg)\toas 0,
$$
where $\toas$ indicates that the limit is understood in the almost sure sense.
\end{thm}

As a direct consequence of this theorem, we obtain a strong law of large numbers for the intrinsic volumes $V_1,\ldots,V_d$ of the projected cube. Here, we recall that the $k$-th intrinsic volume $V_k(K)$, $k\in\{1,\ldots,d\}$, of a convex set $K\subseteq\R^d$ is defined as
$$
V_k(K) := {d\choose k}{\Gamma(1+{k\over 2})\Gamma(1+{d-k\over 2})\over\Gamma(1+{d\over 2})}\E[\vol_k(K|L)],
$$
where $L$ is a random $k$-dimensional linear subspace of $\R^d$ which is distributed on the Grassmannian of all such linear subspaces according to the Haar probability measure, $K|L$ stands for the orthogonal projection of $K$ onto $L$ and $\vol_k$ for the $k$-dimensional Lebesgue measure (in $L$). In particular, if $K$ has dimension $d$, $V_d(K)$ is the $d$-dimensional volume, $2V_{d-1}(K)$ is the surface area of the boundary, and $V_1(K)$ coincides with a constant multiple of the mean width of $K$. For the special case of the volume (corresponding to $k=d$) the next result complements the central limit theorem of Paouris, Pivovarov, and Zinn \cite{PPZ2014}.

\begin{cor}\label{cor:slln volume}
For fixed $d\in\N$ and $k\in\{1,\ldots,d\}$ the $k$-th intrinsic volume $V_k(I^*([-1,1]^n))$ of the projected cube satisfies
\[
\frac{V_k(I^*([-1,1]^n))}{n^{\frac{k}{2}}} \toas \Big({2\over\pi}\Big)^{k/2}V_k(\B_2^d(0,1))=2^{k/2}{d\choose k}{\Gamma(1+{d-k\over 2})\over\Gamma(1+{d\over 2})}\,.
\]
\end{cor}

A central part of this work is to quantify how the randomly projected mass of the cube is distributed inside the ball $\B_2^d(0,\sqrt{2/\pi})$. We do this in the framework of large deviations theory and shall approach this problem in the following more general set-up. We fix a probability measure $\nu$ on $\R$ with everywhere finite moment generating function $M_\nu(x) =\int_\R e^{x u} \nu(\dint u)$, $x\in\R$, and assume that
\begin{equation}\label{eq:MBedingung}
\int_{\R}|\log M_\nu(x)|\,e^{-\varepsilon x^2}\,\dint x<\infty
\end{equation}
for every $\varepsilon>0$. By $\nu^{\otimes n}$ we denote the $n$-fold product probability measure of $\nu$ on $\R^n$. Now, we define the random probability measure
\begin{equation}\label{eq:DefmuN}
\mu^I_{\otimes n} := \nu^{\otimes n} \circ \Big(\frac{1}{\sqrt{n}}I^*\Big)^{-1}
\end{equation}
on $\R^d$, where $I^*:\R^n \to \R^d$ is as in \eqref{eq:defIStern}. In other words, $\mu^I_{\otimes n}$ is a random projection of the product measure $\nu^{\otimes n}$ rescaled by the factor $1/\sqrt n$. This means that
$$
\int_{\R^d}h(x)\,\mu_{\otimes n}^I(\dint x) = \int_{\R^n}h\Big({1\over\sqrt{n}}I^*(x)\Big)\,\nu^{\otimes n}(\dint x)
$$
for all non-negative measurable functions $h:\R^d\to\R$.

Our next result shows that the sequence of random probability measures $(\mu^I_{\otimes n})_{n\in\N}$ satisfies an almost sure large deviations principle (LDP) and identifies the corresponding speed and rate function. For a formal definition of an almost sure LDP we refer to Definition \ref{def:LDP} below, but we mention that roughly speaking this means that, almost surely,
$$
\lim_{n\to\infty}{1\over n}\log\mu_{\otimes n}^I(A) = -\inf_{x\in A}\rate(x)
$$
for all `reasonable' sets $A\subseteq\R^d$, where $\rate$ is called the rate function. 

\begin{thm}\label{thm:as LDP}
Let $\nu$ be a probability measure on $\R$ satisfying \eqref{eq:MBedingung} and let $\mu^I_{\otimes n}$ be defined by \eqref{eq:DefmuN}. Then the sequence $(\mu^I_{\otimes n})_{n\in\N}$ of random probability measures on $\R^d$ satisfies an almost sure LDP with speed $n$ and good rate function $\Lambda^*$ given as the Legendre-Fenchel transform of
\[
\Lambda:\R^d \to\R,\quad t\mapsto \E\big[\log M_{\nu}(\|t\|_2\,g)\big],
\]
where $g\sim\mathcal N(0,1)$ is a standard Gaussian random variable and $\|\,\cdot\,\|_2$ denotes the Euclidean norm.
\end{thm}


\vskip 3mm
The rest of the paper is organized as follows. In Section \ref{sec:special cases}, we take a look at a few examples of distributions $\nu$ on $\R$ and determine in some of the cases the precise rate function $\Lambda^*$ or extract some of its properties. Section \ref{sec:prelim} provides the necessary concepts and results that are used in the proofs. Section \ref{sec:gaussian setting} is devoted to the analysis of large deviations in the Gaussian setting and the transition to the general case as well as the proof of Theorem \ref{thm:as LDP} are presented in Section \ref{sec:general projections}. Finally, in Section \ref{sec:slln}, we present the proof of the strong law of large numbers stated in Theorem \ref{thm:SLLN}.

\section{Some special cases}\label{sec:special cases}

Before we present the proofs of our main results, we discuss a few particular examples of distributions $\nu$ that satisfy \eqref{eq:MBedingung} and for which the function $\Lambda$ in Theorem \ref{thm:as LDP} (or at least some of its properties) can be identified explicitly.

\subsection{The standard Gaussian distribution}

We start by taking for $\nu$ the standard Gaussian distribution $\mathcal{N}(0,1)$. Then, $\nu^{\otimes n}$ is the standard Gaussian distribution on $\R^n$, and its projection $\mu^I_{\otimes n}$ is the standard Gaussian distribution on $\R^d$ rescaled by $1/\sqrt n$. Therefore, the sequence $\mu^I_{\otimes n}$ satisfies a large deviations principle with speed $n$ and rate function
\begin{align*}
\Lambda^*(u) = {\|u\|_2^2\over 2},\qquad u\in\R^d.
\end{align*}
With some effort, one can check that Theorem~\ref{thm:as LDP} gives the same result. Indeed, the moment generating function $M_\nu$ is given by
\begin{align*}
M_\nu(y)={1\over\sqrt{2\pi}}\int_{\R}e^{yx}\,e^{-x^2/2}\,\dint x = e^{y^2/2},\qquad y\in\R.
\end{align*}
Moreover, we can estimate
\begin{align*}
\int_{\R}|\log M_\nu(x)|\,e^{-\varepsilon x^2}\,\dint x = {1\over 2}\int_{\R}x^2\,e^{-\varepsilon x^2}\,\dint x = {\sqrt{\pi}\over 4\varepsilon^{3/2}}<\infty
\end{align*}
for any $\varepsilon>0$, which implies that \eqref{eq:MBedingung} is satisfied for all $\varepsilon>0$. Next, we compute $\Lambda(t)=\E[\log M_\nu(\|t\|_2g)]$ for $t\in\R^d$ as follows:
\begin{align*}
\Lambda(t) &= {1\over\sqrt{2\pi}}\int_{\R}\log\big(e^{\|t\|_2^2y^2/2}\big)e^{-y^2/2}\,\dint y={\|t\|_2^2\over 2}\,{1\over\sqrt{2\pi}}\int_{\R}y^2\,e^{-y^2/2}\,\dint y = {\|t\|_2^2\over 2}\,.
\end{align*}
In this particular situation also the Legendre-Fenchel transform of this function can be computed explicitly and is given by
\begin{align*}
\Lambda^*(u) = \sup_{t\in\R^d}\Big[\langle u,t\rangle_2-{\|t\|_2^2\over 2}\Big] = {\|u\|_2^2\over 2},\qquad u\in\R^d,
\end{align*}
since the supremum is attained at $t=u$.

\subsection{The uniform distribution on $\{-1,1\}$}
As we already know, the random $d$-dimensional projection of the cube $[-1,1]^n$ is close to the Euclidean ball of radius $\sqrt{2n/\pi}$. We now investigate how the projected vertices of the cube are distributed inside this ball. To this end, we apply Theorem~\ref{thm:as LDP} in the special case when $\nu$ is the uniform distribution on $\{-1,1\}$, that is, we take $\nu:=\frac{1}{2}(\delta_{-1}+\delta_{+1})$, where $\delta_{x}$ stands for the Dirac measure at $x \in\R$. Then, $\nu^{\otimes n}$ is the uniform distribution on the discrete cube $\{-1,1\}^n$. We are interested in random $d$-dimensional projections of this distribution, for fixed $d\in\N$ and as $n\to\infty$. The moment generating function of $\nu$ is given by
\[
M_\nu(y) = \int_\R e^{y x}\,\nu(\dint x) = \frac{e^{-y}+e^{y}}{2} = \cosh(y),\qquad y\in\R.
\]
Moreover, using that $\cosh(x)\leq\max\{e^x,e^{-x}\}$ for all $x\in\R$, we see that
\begin{align*}
\int_{\R}|\log M_\nu(x)|e^{-\varepsilon x^2}\,\dint x \leq \int_{-\infty}^0(-x)\,e^{-\varepsilon x^2}\,\dint x+\int_0^{\infty}x\,e^{-\varepsilon x^2}\,\dint x<\infty
\end{align*}
so that \eqref{eq:MBedingung} is satisfied for all $\varepsilon>0$. Since $\Lambda(t)$ depends on $t\in\R^d$ only via the norm $\|t\|_2$, we have $\Lambda(t) = \Psi(\|t\|_2)$, where the function $\Psi:[0,\infty)\to\R$ is given by
$$
\Psi(s) = \E[\log \cosh(sg)],
$$
and $g\sim\mathcal{N}(0,1)$ is a standard Gaussian random variable. The rate function $\Lambda^*$ appearing in Theorem~\ref{thm:as LDP} is then given by $\Lambda^*(t) = \Psi^*(\|t\|_2)$, where $\Psi^*$ is the Legendre-Fenchel transform of $\Psi$. Roughly speaking, Theorem~\ref{thm:as LDP} states that the number of vertices of the cube $[-1,1]^n$ whose projection is ``close'' to the point $t\sqrt n$ is given by $2^n e^{-\Lambda^*(t) n +o(n)}$, for $t\in\R^d$, as $n\to\infty$.

Let us check that $\Psi^*(u) = +\infty$ whenever $u > \sqrt{2/\pi}$ and that $\Psi^*(\sqrt{2/\pi}) = \log 2$.
Indeed, for $s\geq 0$ one has that
\begin{align}\label{eq:021019-1}
\Psi'(s) = \E[g\tanh(gs)] \leq \E|g| = \sqrt{2\over\pi},
\end{align}
for all $s\geq 0$. Moreover,
\begin{align}\label{eq:021019-2}
\notag \Psi(s) = \E\Big[\log {e^{sg}+e^{-sg}\over 2}\Big] &= -\log 2 + \E[|sg|] + \E[\log(1+e^{-2|sg|})]\\
\notag &= -\log 2 + \sqrt{2\over\pi}\,s + \E[\log(1+e^{-2|sg|})]\\
&= -\log 2 + \sqrt{2\over\pi}\,s + o(1),
\end{align}
as $s\to \infty$. If $u>\sqrt{2/\pi}$ then \eqref{eq:021019-2} implies that $us-\Psi(s)\to\infty$, as $s\to\infty$, hence $\Psi^*(u)=+\infty$. On the other hand, if $u=\sqrt{2/\pi}$, then $\sqrt{2/\pi}\,s-\Psi(s)$ is a non-decreasing function of $s$ by \eqref{eq:021019-1}, which implies that its supremum is attained as $s\to\infty$ and hence equals $\log 2$ in view of \eqref{eq:021019-2}. 

\subsection{The uniform distribution on $[-1,1]$}

Let us now consider for $\nu$ the continuous analogue of the discrete uniform distribution studied in the previous section. That is, we let $\nu$ be the uniform distribution on the interval $[-1,1]$ whose density is $1/2$ on $[-1,1]$ and zero otherwise.  Then, $\nu^{\otimes n}$ is the uniform distribution on the cube $[-1,1]^n$ and we are interested in its random $d$-dimensional projections.  Clearly,
\[
M_\nu(y) = \int_\R e^{y x}\,\nu(\dint x) = \frac{1}{2}\int_{-1}^{1} e^{y x}\,\dint x = \frac{\sinh(y)}{y},\qquad y\in\R.
\]
Since $\sinh(x)\leq e^{x}/2$ for $x>0$, we also have that, for any $\varepsilon>0$,
\begin{align*}
\int_{\R}|\log M_\nu(x)|e^{-\varepsilon x^2}\,\dint x \leq -2\int_{\R}\log(x)\,e^{-\varepsilon x^2}\,\dint x + 2\int_{\R}(x-\log 2)\,e^{-\varepsilon x^2}\,\dint x <\infty
\end{align*}
so that condition \eqref{eq:MBedingung} is satisfied. We have $\Lambda(t) = \Psi(\|t\|_2)$, where $\Psi:[0,\infty)\to\R$ is given by
$$
\Psi(s) = \E\Big[\log{\frac{\sinh (sg)}{sg}}\Big].
$$
Again, the rate function $\Lambda^*$ appearing in Theorem~\ref{thm:as LDP} is given by $\Lambda^*(t) = \Psi^*(\|t\|_2)$.
Let us check that, as in the discrete case, $\Psi^*(u)= + \infty$ if $u > \sqrt{2/\pi}$. Indeed, we have
\begin{align*}
\Psi(s) = \E\Big[\log{e^{sg}-e^{-sg}\over 2sg}\Big] &= -\log 2-\log s-\E\log|g|+\E[|sg|]+\E[\log(1-e^{-2|sg|})]\\
&=-{1\over 2}(\log 2+\gamma)-\log s+\sqrt{2\over\pi}\,s+\E[\log(1-e^{-2|sg|})]\\
&=-{1\over 2}(\log 2+\gamma)-\log s+\sqrt{2\over\pi}\, s+o(1)\,,
\end{align*}
as $s\to \infty$, where $\gamma$ is the Euler-Mascheroni constant. By the same reasoning as in the previous section it follows that $\Psi^*(u)= + \infty$ whenever $u > \sqrt{2/\pi}$. On the other hand, since the expression for $\Psi(s)$ contains the additional term $-\log s$, we have that $\Psi^*(\sqrt{2/\pi})=+\infty$ in contrast to the discrete case.

\section{Preliminaries}\label{sec:prelim}

\subsection{General notation}

We start by introducing some notation that is used throughout the paper. We fix a space dimension $d\geq 1$ and denote by $\|\,\cdot\,\|_2$ the Euclidean norm and by $\langle\,\cdot\,,\,\cdot\,\rangle_2$ the Euclidean scalar product in $\R^d$. For $x\in\R^d$ and $r>0$ we denote by $\B_2^d(x,r)$ the Euclidean ball of radius $r$ centred at $x$, and put $\B_2^d:=\B_2^d(0,1)$. Similarly, we denote by $\|\,\cdot\,\|_\infty$ the maximum norm on $\R^d$ given by $\|x\|_\infty=\max\{|x_1|,\ldots,|x_d|\}$ for $x=(x_1,\ldots,x_d)\in\R^d$.  For a subset $A$ of a topological space we write $A^\circ$ and $\overline{A}$ for the interior and the closure of $A$, respectively.

By $(\Omega,\mathcal{F},\Pro)$ we denote our underlying probability space, which is implicitly assumed to be rich enough to carry all the random objects we deal with in this paper. By $\E$ we denote expectation (that is, integration) with respect to $\Pro$. If more than one random object $X$ is involved we shall use the notation $\E_X$ to indicate that the expectation is taken only with respect to $X$. By $\mathcal{N}(0,1)$ we denote the standard Gaussian distribution on $\R$ and write $g\sim\mathcal{N}(0,1)$ to say that the random variable $g$ has distribution $\mathcal{N}(0,1)$. In addition, if two random variables $X$ and $Y$ have the same distribution we indicate this by writing $X\overset{\dint}{=}Y$. The almost sure convergence of a sequence of random elements $X_n$ to another random element $X$ is indicated by $X_n\toas X$.

The group of orthogonal $d\times d$ matrices will be denoted by $\mathcal{O}(d)$ and we shall write $A^*$ for the adjoint of a matrix $A$. Also, by $\id_{d\times d}$ we denote by $d\times d$ identity matrix and by $\id_{\R^d}:\R^d\to\R^d$ the identity map. For a linear map $T:\R^d\to\R^n$ (for some $n\in\N$) we let $T^*$ be the adjoint operator satisfying $\langle Tx,y\rangle_2=\langle x,T^*y\rangle_2$ for all $x\in\R^d$ and $y\in\R^n$. Moreover, we denote by $\|T\|_{\rm op}$ the operator norm of $T$, which is given by $\|T\|_{\rm op}:=\sup\{\|T(x)\|_2:\|x\|_2\leq 1\}$.

For $a<b$ we denote by $\mathscr{C}[a,b]$ for the space of continuous functions $f:[a,b]\to\R$ endowed with the supremum norm $\|f\|_\infty:=\sup\{|f(x)|:x\in[a,b]\}$.

\subsection{Stiefel Manifolds}

For $n,d\in\N$ with $d\leq n$, the Stiefel manifold $\mathbb V_{n,d}$ (over $\R$) is defined as the set of all orthonormal $d$-frames in $\R^n$, i.e., the set of all ordered $d$-tuples of orthonormal vectors in Euclidean space $\R^n$. Alternatively, the Stiefel manifold can be thought of as the set of $n\times d$ matrices and a $d$-frame $u_1,\ldots,u_d$ is represented as a matrix with the $d$ columns $u_1,\ldots,u_d$. Formally, this means that
\[
\mathbb V_{n,d} = \big\{ M \in\R^{n\times d}\,:\, A^*A = \id_{d\times d}  \big\}.
\]
We denote by $\mu_{n,d}$ the Haar probability measure on the Stiefel manifold $\mathbb V_{n,d}$, i.e., the unique probability measure on $\mathbb V_{n,d}$ which is invariant under the {two-sided action of the product of the orthogonal groups $\mathcal{O}(n)\times \mathcal O(d)$.} We frequently refer to this measure as the uniform distribution on the Stiefel manifold. So, if $A$ is a random matrix uniformly distributed on $\mathbb V_{n,d}$, then $UAV\stackrel{\dint}{=} A$ for every $U\in\mathcal O(n)$ and $V\in\mathcal O(d)$. The next result shows how to generate a uniform random element in $\mathbb{V}_{n,d}$ using Gaussian random matrices. We include the proof for the sake of completeness.

\begin{lemma}\label{lem:uniform distribution Stiefel manifold}
Let $d,n\in\N$ and assume that $d \leq n$. Consider a Gaussian random matrix $G=(g_{ij})_{i,j=1}^{d,n}:\R^n\to\R^d$ with independent standard normal entries. Then the random matrix $I = G^*(GG^*)^{-1/2}:\R^d\to\R^n$ is uniformly distributed on the Stiefel manifold $\mathbb V_{n,d}$.
\end{lemma}
\begin{rmk}
The polar decomposition (see, e.g., \cite[Theorem 3.5]{EG2015}) states that for any linear operator $T:\R^n\to\R^d$, $n\geq d$ there exists a linear isometry $J :\R^d\to\R^n$ (i.e., an isometric embedding, or, equivalently, an isometric isomorphism onto its image) such that
\[
T^* \,=\, J (TT^*)^{1/2}.
\]
In our setting, taking $T=G$, $I$ is the isometric embedding associated with the operator $G$. 
\end{rmk}
\begin{proof}[Proof of Lemma~\ref{lem:uniform distribution Stiefel manifold}.]
We show that $I$ is uniformly distributed on the Stiefel manifold $\mathbb V_{n,d}$. In order to do this, we prove left and right orthogonal invariance, i.e., we show that for all $U\in\mathcal O(n)$ and $V\in \mathcal O(d)$,
\[
U I V \stackrel{\dint}{=} I\,.
\]
For $U\in\mathcal O(n)$ and $V\in \mathcal O(d)$, we define a random linear operator
\[
\widetilde G := V^* G U^* :\R^n \to\R^d.
\]
{By computing the covariances between the entries of $\widetilde G$, one easily checks that  $\widetilde G \stackrel{\dint}{=} G$.} For the random linear operator $\widetilde I := \widetilde G^*(\widetilde G\widetilde G^*)^{-1/2}$, we obtain
\[
\widetilde I = U G^* V (V^* GU^*UG^* V)^{-1/2} = U G^* V (V^* GG^* V)^{-1/2}\,.
\]
Now, if $A:= (GG^*)^{1/2}$, meaning that $A=A^*$ is positive semi-definite and $A^2= GG^*$, then
\[
V^*A(V^*)^{-1} = (V^*GG^*V)^{1/2},
\]
because $(V^*A(V^*)^{-1})^* = V^*A(V^*)^{-1}$ is positive semi-definite and
$$
(V^*A(V^*)^{-1})(V^*A(V^*)^{-1}) = V^*A^2(V^*)^{-1}=V^*GG^*V.
$$
Therefore,
\begin{align*}
\widetilde I  = UG^*V \big(V^*(GG^*)^{1/2}(V^*)^{-1}\big)^{-1}
  = UG^*VV^*(GG^*)^{-1/2}V
  = UG^*(GG^*)^{-1/2}V = U I V.
\end{align*}
Since $G\stackrel{\dint}{=}\widetilde G$, we also have $\widetilde I \stackrel{\dint}{=} I$. Therefore, we conclude that
\[
U I V \stackrel{\dint}{=} I,
\]
which shows the left and right invariance. We conclude that $I$ is uniformly distributed on the Stiefel manifold $\mathbb V_{n,d}$.
\end{proof}

\subsection{Strong law of large numbers in Banach spaces}

The following result is the strong law of large numbers for random elements taking values in a separable Banach space. It extends the classical strong law of large numbers for real-valued random variables of Kolmogorov to the Banach space set-up. The result can be found, for instance, in the monograph \cite[Corollary 7.10]{LT2011}. 

\begin{proposition}[SLLN in Banach spaces]\label{prop:SLLN BS}
Let $X$ be a random variable taking values in a separable Banach space $(F,\|\,\cdot\,\|_F)$ and $(X_n)_{n\in\N}$ a sequence of independent copies of $X$. If and only if $\E[\|X\|_F] < \infty$ one has that
\[
{1\over n}\sum_{i=1}^n X_i\toas\E[X].
\]
\end{proposition}

We remind the reader that in the previous proposition $\E[X]=\int_F X\,\dint\Pro$ stands for the Pettis integral of the $F$-valued random variable $X$, which we think of being defined on the probability space $(\Omega,\mathcal{F},\Pro)$.

\subsection{Random measures}

Consider a Polish space $S$. We denote by $\mathcal M_1(S)$ the space of (Borel) probability measures on $S$ and supply $\mathcal M_1(S)$ with the $\sigma$-algebra $\mathscr B(\mathcal M_1(S))$ generated by the evaluation mappings $\varepsilon_B: \mu\mapsto \mu(B)$, where $\mu\in \mathcal M_1(S)$ and $B\in\mathscr B(S)$, the Borel $\sigma$-algebra on $S$. This means that $\mathscr B(\mathcal M_1(S))$ is the smallest $\sigma$-algebra for which all the mappings $\varepsilon_B$ become measurable. {Endowed} with the weak topology, the space $\mathcal M_1(S)$ is a Polish space and $\mathscr B(\mathcal M_1(S))$ coincides with the Borel-$\sigma$ generated by the weak topology (see, e.g., \cite{KallenbergRM}). A random measure $\nu$ on $S$ is a random element in the measurable space $(\mathcal M_1(S),\mathscr B(\mathcal M_1(S)))$, i.e., a measurable mapping $\nu:\Omega \to \mathcal M_1(S)$, where we recall that $(\Omega,\mathcal{F},\Pro)$ is our underlying probability sapce.

\subsection{Large deviations and the G\"artner-Ellis theorem}

Let us briefly present the necessary background material from the theory of large deviations, which may be found in \cite{DZ, dH, DS1989}, for example. We start with the definition of full and weak large deviations principles.

\begin{df}[Full and weak LDP]\label{def:LDP}
Let $(\mu_n)_{n\in\N}$ be a sequence of probability measures on $\R^d$, $(s_n)_{n\in\N}$ a positive sequence such that $s_n\uparrow +\infty$, and $\rate:\R^d\to[0,+\infty]$ a lower semi-continuous function.
We say that $(\mu_n)_{n\in\N}$ satisfies a full large deviations principle (full LDP) with speed $s_n$ and rate function $\rate$ if
\begin{equation}\label{eq:LDPdefinition}
\begin{split}
-\inf_{x\in A^\circ}\rate(x)
\leq\liminf_{n\to\infty}{1\over s_n}\log \mu_n(A)
\leq\limsup_{n\to\infty}{1\over s_n}\log\mu_n(A) \leq -\inf_{x\in\overline{A}}\rate(x)
\end{split}
\end{equation}
for all Borel sets $A\subseteq \R^d$. The rate function $\rate$ is called good if its level sets  $\{x\in \R^d\,:\, \rate(x) \leq \alpha \}$ are compact for all $\alpha\geq 0$.

We say that $(\mu_n)_{n\in\N}$ satisfies a weak LDP with speed $s_n$ and rate function $\rate$ if the upper bound in \eqref{eq:LDPdefinition} is valid only for compact sets $A\subseteq \R^d$.
\end{df}

We note that \eqref{eq:LDPdefinition} is equivalent to the following two conditions:
\begin{itemize}
\item[(i)] $-\inf\limits_{x\in U}\rate(x) \leq \liminf\limits_{n\to\infty}\frac{1}{n}\log \mu_n(U)$ for all open sets $U\subseteq\R^d$,
\item[(ii)] $\limsup\limits_{n\to\infty}\frac{1}{n}\log \mu_n(C) \leq -\inf\limits_{x\in C}\rate(x)$ for all closed sets $C\subseteq\R^d$.
\end{itemize}

The notions of weak and full LDPs are separated by the concept of exponential tightness of the sequence of probability measures (see, e.g., \cite[Lemma 1.2.18]{DZ} and \cite[Lemma 27.9]{Kallenberg}).

\begin{proposition}\label{prop:equivalence weak and full LDP}
Let $(\mu_n)_{n\in\N}$ be a sequence of probability measures on $\R^d$ satisfying a weak LDP with speed $s_n$ and rate function $\rate$. Then $(\mu_n)_{n\in\N}$ satisfies a full LDP with good rate function $\rate$ if and only if the sequence is exponentially tight, i.e., if and only if for every $M\in(0,\infty)$ there exists a compact set $K\subseteq \R^d$ such that
$$
\limsup_{n\to\infty}{1\over s_n}\log \mu_n(\R^n\setminus K)\leq -M\,.
$$
\end{proposition}

In this manuscript, we shall prove an almost sure version of a large deviations principle.
{In this setting, the probability measures $\mu_n$ in Definition \ref{def:LDP} are random and we require that \eqref{eq:LDPdefinition} (or the two equivalent conditions (i) and (ii) above) hold for almost all realizations of the sequence $(\mu_n)_{n\in\N}$.}

We shall use a special case of the G\"artner-Ellis theorem in our argument. To rephrase it, assume that $(\mu_n)_{n\in\N}$ is a sequence of probability measures on $\R^d$ with moment generating functions
\[
\varphi_n(t) := \int_{\R^d} e^{\langle x,t\rangle_2}\,\mu_n(\dint x),\qquad t\in \R^d,\,\,n\in\N.
\]
The Legendre-Fenchel transform of a function $F:\R^d\to\R$ is denoted by $F^*$ and is given by $F^*(x)=\sup\{\langle x,t\rangle_2-F(t):t\in\R^d\}$.

\begin{proposition}[G\"artner-Ellis theorem]\label{prop: gaertner-ellis thm}
Let $(\mu_n)_{n\in\N}$ be a sequence of probability measures on $\R^d$ with moment generating functions $\varphi_n$, $n\in\N$.
Assume that the limit
\[
\Lambda(t) := \lim_{n\to\infty} \frac{1}{n} \log \varphi_n(nt)
\]
exists finitely for all $t\in\R^d$  and that $\Lambda$ is differentiable on $\R^d$. Then $(\mu_n)_{n\in\N}$ satisfies a large deviations principle on $\R^d$ with good rate function $\Lambda^*$.
\end{proposition}

\section{The case of Gaussian projections}\label{sec:gaussian setting}

Passing between Gaussian random matrices and random orthogonal projections in the Grassmannian sense has demonstrated to be useful in a variety of contexts as can be seen, for example, in \cite{AS1992, BV1994, BH1999, DT2010, MTJ2001, PPZ2014} and references cited therein. This will also be the spirit in our work, starting with the results for `Gaussian projections' followed by a transition step that allows us to go over to random orthogonal projections. Of course, such a transition requires a careful analysis.

Let us briefly describe the Gaussian setting. In what follows, $\nu$ will be a probability measure on $\R$ with everywhere finite moment generating function $M_\nu$ satisfying \eqref{eq:MBedingung} for all $\varepsilon>0$. For $n\in\N$ we define the random probability measure
\[
\mu^G_{\otimes n} := \nu^{\otimes n} \circ \Big(\frac{1}{n}G\Big)^{-1}
\]
on $\R^d$, where $G=(g_{ij})_{i,j=1}^{d,n}:\R^n\to\R^d$ is a random Gaussian matrix with independent standard Gaussian entries. Our next result is an almost sure LDP for the sequence of random probability measures $\mu^G_{\otimes n}$. It is the first step in the proof of our main result, Theorem \ref{thm:as LDP}.

\begin{proposition}[LDP in the Gaussian setting]\label{prop:ldp gaussian setting}
The sequence $(\mu^G_{\otimes n})_{n\in\N}$ of random probability measures on $\R^d$ satisfies an almost sure LDP with speed $n$ and good rate function $\Lambda^*$ given as the Legendre-Fenchel transform of
\[
\Lambda:\R^d \to\R,\quad t\mapsto \E\big[\log M_{\nu}(\|t\|_2\,g)\big],
\]
where $g\sim\mathcal N(0,1)$ is a standard Gaussian random variable.
\end{proposition}
\begin{proof}
{For every $n\in\N$ we let $\varphi_n(t):= \int_{\R^d} e^{\langle x, t \rangle_2} \,\mu^G_{\otimes n}(\dint x)$, $t\in\R^d$, denote the  moment generating function of the random probability measure $\mu_{\otimes n}^{G}$. Note that $\varphi(t)$ is random because so is $G$.
We have}
\begin{align*}
\log \varphi_n(nt) & = \log \int_{\R^d} e^{\langle x, n t \rangle_2} \,\mu^G_{\otimes n}(\dint x) = \log \int_{\R^n} e^{\langle (n^{-1}G)x,n t \rangle_2}\,\nu^{\otimes n}(\dint x) = \log\int_{\R^n} e^{\langle Gx,t\rangle_2}\,\nu^{\otimes n}(\dint x),
\end{align*}
for $t\in\R^d$. For $n\in\N$, let $X^{(n)}:=(X_1,\dots,X_n)$ be a random vector with independent coordinates distributed according to $\nu$. Also assume that $X^{(n)}$ is independent from the Gaussian random matrices $G$ (and recall the dependence of $G$ on $n$, which is suppressed in our notation). Then, for every $t\in\R^d$ and all $n\in\N$, we can write
\begin{align*}
\log \varphi_n(nt) & = \log \E_{X_1,\dots,X_n}\big[ e^{\langle GX, t \rangle_2}\big] \cr
& = \log \E_{X_1,\dots,X_n}\bigg[ \exp\bigg(\sum_{i=1}^d t_i \sum_{j=1}^n g_{ij}X_j\bigg)\bigg] \cr
& = \log \E_{X_1,\dots,X_n}\bigg[ \exp\bigg(\sum_{j=1}^n X_j  \sum_{i=1}^d t_i g_{ij} \bigg)\bigg] \cr
& = \log \prod_{j=1}^n \E_{X_j}\bigg[ \exp\bigg(X_j  \sum_{i=1}^d t_i g_{ij} \bigg)\bigg]\,.
\end{align*}
Since $\sum_{i=1}^d t_i g_{ij}$ has the same distribution as $\|t\|_2\, g_j$, where $g_1,\dots,g_n$ are independent standard Gaussian random variables (that are assumed to be independent of $X_1,\ldots,X_n$), we obtain
\begin{align*}
\log \varphi_n(nt)
& = \sum_{j=1}^n \log \E_{X_j}\bigg[ \exp\bigg(\|t\|_2\, g_jX_j \bigg)\bigg] \cr
& = \sum_{j=1}^n \log M_{X_1}\big(\|t\|_2\, g_j\big),
\end{align*}
where $M_{X_1}(\lambda) = \E\big[e^{\lambda X_1}\big] = M_\nu(\lambda)$, $\lambda\in\R$, is the moment generating function of $X_1$.

Our aim is to show that $\frac 1n \log \varphi_n(nt)$ converges almost surely, as $n\to\infty$.  To prove this, we shall use the strong law of large numbers in the Banach space $\mathscr C[-r,r]$ of continuous functions on the interval $[-r,r]$, $r\in(0,\infty)$ equipped with the $\|\cdot\|_\infty$-norm (see Proposition \ref{prop:SLLN BS}). Consider the random element $Z:\Omega \to \mathscr C[-r,r]$, which assigns to each $\omega\in\Omega$ the random continuous function
$$
[-r,r]\ni x\mapsto Z(\omega)(x)=\log M_{X_1} \big(x  g(\omega)\big),
$$
where $g\sim\mathcal N(0,1)$ is a standard Gaussian random variable.
In order to apply the strong law of large numbers in this setting, we have to show that
\[
\E\|Z\|_\infty=\E \max_{x\in [-r,r]}|Z(x)| < \infty.
\]
Since moment generating functions are $\log$-convex (where finite), we know that $\log M_{X_1}$ is a convex function on $\R$. This means that for every $\omega\in\Omega$,
\[
\max_{x\in [-r,r]}|Z(\omega)(x)| \in \{|Z(\omega)(-r)|, |Z(\omega)(r)| \} \leq |Z(\omega)(-r)| + |Z(\omega)(r)|.
\]
This estimate together with the symmetry of Gaussian random variables and the substitution $y=rx$ implies that
\begin{align*}
\E\Big[\max_{x\in [-r,r]}|Z(x)| \Big] &\leq 2 \,\E\big[|Z(r)|\big] \\
&= \frac{2}{\sqrt{2\pi}} \int_{\R} |\log M_{X_1}(r x)|e^{-x^2/2}\,\dint x\\
&= {1\over r}\sqrt{2\over\pi}\int|\log M_\nu(y)|\,e^{-{1\over r^2}{y^2\over 2}}\,\dint y<\infty,
\end{align*}
where the finiteness follows from assumption \eqref{eq:MBedingung} on the measure $\nu$. Therefore, we obtain from the strong law of large numbers for random elements in separable Banach spaces (Proposition \ref{prop:SLLN BS}) that with probability $1$,
\[
\frac{1}{n}\sum_{j=1}^n \log M_{X_1}\big( \|t\|_2\, g_j\big) \ton \E_{g_1}\big[\log M_{X_1}\big(\|t\|_2\, g_1\big) \big],
\]
uniformly for all $t$ with $\|t\|_2 \leq r$. This means that we can find a subset $\Omega_1$ of our probability space $\Omega$ such that $\Pro(\Omega_1)=1$ and for every $\omega\in \Omega_1$ the corresponding realizations of the moment generating functions satisfy
\[
\frac{1}{n}\log \varphi_n(nt) \ton \Lambda(t):= \E_{g_1}\big[\log M_{X_1}\big(\|t\|_2\, g_1\big) \big]
\]
uniformly for all $t$ with $\|t\|_2 \leq r$. Since $r\in(0,\infty)$ was arbitrary, we have for almost all $\omega\in\Omega$  that the limit
\[
\lim_{n\to\infty} \frac{1}{n}\log \varphi_n(\omega)(nt) = \Lambda(t)
\]
exists for all $t\in\R^d$.

We claim that the function $\Lambda$ is differentiable on $\R^d$. To prove this, we can assume without loss of generality that $\E[X_1]=0$. Indeed, if $\E[X_1]=m\neq 0$ we can write $X_1=Y_1+m$ and note that
\begin{align*}
\Lambda(t) &= \E_{g_1}\big[\log M_{X_1}\big(\|t\|_2\, g_1\big) \big]\\
& = \E_{g_1}\big[m\|t\|_2g_1+\log M_{Y_1}\big(\|t\|_2\, g_1\big) \big] 
=\E_{g_1}\big[\log M_{Y_1}\big(\|t\|_2\, g_1\big) \big],
\end{align*}
so that we can replace $X_1$ by its centered version $Y_1$.

Let us first prove that $\Lambda$ is differentiable on $\R^d\setminus\{0\}$. We define the (infinitely differentiable) function $f(z):=\log M_{X_1}(z)$, $z\in\R$. What we need to show is that the function
$$
r\mapsto {1\over\sqrt{2\pi}}\int_{\R} f(rz)\,e^{-z^2/2}\,\dint z, \qquad r>0,
$$
is differentiable on $(0,\infty)$. This follows from the differentiation rule under the integral sign~\cite[Theorem~6.28]{klenke_book} once we prove that, for all $r>0$ and every interval $[a,b]\subseteq(0,\infty)$,
\begin{equation*}
{1\over\sqrt{2\pi}}\int_{\R} \sup_{r\in [a,b]} |z f'(rz)|\, e^{-z^2/2}\,\dint z <\infty.
\end{equation*}
Since the function $f$ is convex, the supremum is attained either at $a$ or at $b$ and it suffices to prove that for all $r>0$,
\begin{equation}\label{eq:090919a}
{1\over\sqrt{2\pi}}\int_{\R}  |z f'(rz)|\, e^{-z^2/2}\,\dint z <\infty.
\end{equation}
The differentiability of $f$ implies that for each $M>0$, $\sup_{y\in[-M,M]}|f'(y)|=:C_M<\infty$. Thus, splitting the integral into two parts, we have that, for each $M>0$,
\begin{align*}
&{1\over\sqrt{2\pi}}\int_{\R} \big|zf'(rz)\big|\,e^{-z^2/2}\,\dint z\\
&\leq  {1\over\sqrt{2\pi}}\int_{[-M/r,M/r]}\big|zf'(rz)\big|\,e^{-z^2/2}\,\dint z+ {1\over\sqrt{2\pi}}\int_{\R\setminus[-M/r,M/r]}\big|zf'(rz)\big|\,e^{-z^2/2}\,\dint z\\
&\leq {C_{M}\over\sqrt{2\pi}}\int_{\R}|z|e^{-z^2/2}\,\dint z+ {1\over\sqrt{2\pi}}\int_{\R\setminus[-M,M]}\Big|{t\over r}f'(t)\Big|\,e^{-t^2/(2r^2)}\,{\dint t\over r} \\
&=\sqrt{2\over\pi}\,C_M+ {1\over\sqrt{2\pi}}\,{1\over r^2}\int_{\R\setminus[-M,M]}|f'(t)|\,e^{-t^2/(2r^2)+\log|t|}\,\dint t.
\end{align*}
Let $M_1\in(0,\infty)$ be such that $e^{-t^2/(2r^2)+\log|t|}\leq e^{-t^2/(4r^2)}$ whenever $|t|>M_1$. Then 
\begin{align*}
\int_{\R\setminus[-M,M]}|f'(t)|\,e^{-t^2/(2r^2)+\log|t|}\,\dint t &\leq \int_{\R}|f'(t)|\,e^{-t^2/(4r^2)}\,\dint t\\
&=\int_{0}^\infty|f'(t)|\,e^{-t^2/(4r^2)}\,\dint t + \int_{-\infty}^0|f'(t)|\,e^{-t^2/(4r^2)}\,\dint t,
\end{align*}
whenever $M>M_1$. Since $f'(0)=0$ (recall that we assume that $\E[X_1]=0$) we have $f'(t)\geq 0$ for $t\geq 0$ and $f'(t)\leq 0$ if $t\leq 0$. 
Applying now integration-by-parts to the first integral shows that
\begin{align*}
\int_{0}^\infty|f'(t)|\,e^{-t^2/(4r^2)}\,\dint t = \int_{0}^\infty f'(t)\,e^{-t^2/(4r^2)}\,\dint t &= {1\over 2r^2}\int_{0}^\infty f(t)\,t\,e^{-t^2/(4r^2)}\,\dint t.
\end{align*}
The last expression is finite by our assumption \eqref{eq:MBedingung}, since $te^{-t^2/(4r^2)}=O(e^{-t^2/(8r^2)})$. As the second integral can be handled in the same way, we conclude that $\Lambda$ is differentiable on $\R^d\setminus\{0\}$.

It remains to prove that $\Lambda$ is differentiable at the origin of $\R^d$. In fact, we shall show that its differential at the origin vanishes. To this end we need to prove that
$$
\E_{g_1}[\log M_{X_1}(\|t\|_2g_1)] = o(\|t\|_2),
$$
as $\|t\|_2\to 0$. This is equivalent to
$$
\lim_{r\downarrow 0}{\E_{g_1}[f(rg_1)]\over r} = 0.
$$
Assuming we can interchange limit and expectation we can write
$$
\lim_{r\downarrow 0}{\E_{g_1}[f(rg_1)]\over r} = \E_{g_1}\Big[\lim_{r\downarrow 0}{f(rg_1)\over r}\Big] = \E_{g_1}[f'(0)g_1] = 0.
$$
To verify the assumptions of the dominated convergence theorem observe that for all $|r|\leq 1$, it follows from the intermediate value theorem that
$$
\Big|{f(rg_1)\over r}\Big| = \Big|{f(rg_1)-f(0)\over rg_1}\,g_1\Big| = |f'(\xi)g_1|
$$
for some random variable $\xi$ satisfying $|\xi|\leq|g_1|$. By the convexity of $f$ we have that $|f'(\xi)|\leq|f'(g_1)|+|f'(-g_1)|$ and the random variables $|f'(\pm g_1)g_1|$ are in fact integrable by the same arguments as we already used above. This finally completes the proof of the differentiability of the function $\Lambda$ on $\R^d$.

Having established differentiability, we can now apply the G\"artner-Ellis theorem (see Proposition \ref{prop: gaertner-ellis thm}). From this we obtain for every $\omega \in \Omega_1$ that the corresponding realization $\big(\mu_{\otimes n}^{G(\omega)}\big)_{n\in\N}$ satisfies a large deviations principle with speed $n$ and good rate function $\Lambda^*$ being the Legendre-Fenchel transform of
\[
\Lambda:\R^d \to\R,\quad t\mapsto \E_{g_1}\big[\log M_{X_1}(\|t\|_2g_1)\big].
\]
In other words, this means that the sequence of random probability  measures $\big(\mu_{\otimes n}^{G}\big)_{n\in\N}$ on $\R^d$ satisfies the desired almost sure LDP.
\end{proof}

\section{The case of uniform random projections}\label{sec:general projections}

Let $n,d\in\N$ with $d\leq n$ as in the previous section. In what follows,
$\nu$ will be any probability measure on $\R$ satisfying \eqref{eq:MBedingung} and $\nu^{\otimes n}$ shall denote its $n$-fold product probability measure on $\R^n$. We want to prove an almost sure LDP for the sequence of random probability measures
\[
\mu^I_{\otimes n} := \nu^{\otimes n} \circ \Big(\frac{1}{\sqrt{n}}I^*\Big)^{-1}
\]
on $\R^d$, where $I^*:\R^n \to \R^d$ is given by $I^* = (GG^*)^{-1/2}G$ and $G=(g_{ij})_{i,j=1}^{d,n}$ is a random Gaussian random matrix with independent standard Gaussian entries. From Lemma \ref{lem:uniform distribution Stiefel manifold} we know that $I$ is uniformly distributed on the Stiefel manifold $\mathbb V_{n,d}$. 

\begin{proof}[Proof of Theorem \ref{thm:as LDP}]
We start by observing that for any measurable set $A\subseteq \R^d$ (we use the notation that has just been introduced) and each $\omega\in\Omega$ we have
\begin{align}\label{eq:21-06-1}
\nonumber\mu^I_{\otimes n} (A) & = \bigg(\nu^{\otimes n} \circ \Big(\frac{1}{\sqrt{n}}I^*\Big)^{-1}\bigg)(A) = \bigg(\nu^{\otimes n}\circ\Big(\frac{1}{\sqrt{n}}(GG^*)^{-1/2}G\Big)^{-1}\bigg)(A) \\
\nonumber& = \bigg(\nu^{\otimes n}\circ\Big(\frac{n}{\sqrt{n}}(GG^*)^{-1/2}\frac{1}{n}G\Big)^{-1}\bigg)(A) \\
& = \bigg(\nu^{\otimes n} \circ \Big(\frac{1}{n}G\Big)^{-1}\bigg)\Big( \frac{1}{\sqrt{n}}(GG^*)^{1/2}(A)\Big) = \mu^G_{\otimes n}\Big( \frac{1}{\sqrt{n}}(GG^*)^{1/2}(A)\Big)\,.
\end{align}
The matrix $(GG^*)^{1/2}:\R^d\to\R^d$ is invertible with probability one. Moreover, the entries of $GG^*$ are simply the inner products of rows of $G$, i.e.,
\[
GG^* = \Big(\big\langle (g_{ik})_{k=1}^n, (g_{jk})_{k=1}^n \big\rangle_2 \Big)_{i,j=1}^d = \Big(\sum_{k=1}^n g_{ik}g_{jk}\Big)_{i,j=1}^d\,.
\]
We therefore obtain from the classical strong law of large numbers that
\[
\frac{1}{n}GG^* = \Big(\frac{1}{n}\sum_{k=1}^n g_{ik}g_{jk}\Big)_{i,j=1}^d \toas \id_{d\times d}.
\]
In particular, this implies that
\begin{align}\label{eq:almost sure convergence of scaled modulus to identity}
\frac{1}{\sqrt{n}}(GG^*)^{1/2}\toas \id_{d\times d}
\end{align}
as well. We will use the latter fact together with the almost sure LDP for $\big(\mu_{\otimes n}^{G}\big)_{n\in\N}$ provided in Proposition \ref{prop:ldp gaussian setting} to deduce an almost sure weak LDP for the sequence $(\mu^I_{\otimes n})_{n\in\N}$.
\vskip 2mm
\noindent\emph{Weak LDP -- upper bound:} Let $K\subseteq \R^d$ be a compact set. Let $R\in(0,\infty)$ be such that $K\subseteq \B_2^d(0,R)$. As a consequence of \eqref{eq:almost sure convergence of scaled modulus to identity}, we know that on a set $\Omega_2\subseteq \Omega$ with $\Pro(\Omega_2)=1$ the following holds. For every $\omega\in\Omega_2$ and $\varepsilon\in(0,\infty)$ there exists some $N=N(\varepsilon,\omega)\in\N$ such that for all $n\geq N$,
\[
\Big\| \frac{1}{\sqrt{n}}(GG^*)^{1/2} - \id_{d\times d} \Big\|_{\rm op} \leq \varepsilon.
\]
This implies that for large enough $n\in\N$,
\begin{align}\label{eq:inclusion}
\frac{1}{\sqrt{n}}(GG^*)^{1/2}(K) \subseteq K + \varepsilon R\,\B_2^d\,,
\end{align}
on the event $\Omega_2$.  In particular, as a Minkowski sum of two compact sets the latter set is again compact. Recalling from the proof of Proposition \ref{prop:ldp gaussian setting} that $\Omega_1\subseteq \Omega$ is the event with $\Pro(\Omega_1)=1$ on which the almost sure LDP in the Gaussian setting holds, we obtain on $\Omega_1\cap\Omega_2$ that
\begin{eqnarray*}
\limsup_{n\to\infty} \frac{1}{n} \log \mu^I_{\otimes n}(K) & = & \limsup_{n\to\infty} \frac{1}{n} \log \mu^G_{\otimes n}\Big( \frac{1}{\sqrt{n}}(GG^*)^{1/2}(K)\Big) \cr
& \leq & \limsup_{n\to\infty} \frac{1}{n} \log\mu^G_{\otimes n}\big( K + \varepsilon R\,\B_2^d\big) \cr
& \leq & -\inf_{x\in K + \varepsilon R\,\B_2^d}\Lambda^*(x)\,,
\end{eqnarray*}
where we used \eqref{eq:21-06-1}, \eqref{eq:inclusion} and Proposition \ref{prop:ldp gaussian setting} in this order.
Taking the limit, as $\varepsilon\downarrow 0$, and using the lower semi-continuity of the rate function $\Lambda^*$, we obtain the desired weak LDP upper bound
\[
\limsup_{n\to\infty} \frac{1}{n} \log \mu^I_{\otimes n}(K) \leq -\inf_{x\in K}\Lambda^*(x)
\]
on the event $\Omega_1\cap\Omega_2$.
\vskip 2mm
\noindent\emph{Weak LDP -- lower bound:} We start with the case of a bounded open set $U\subseteq \R^d$. Let us define for $\varepsilon>0$ the $\varepsilon$-interior of $U$ as
\[
U^{\circ \varepsilon} := \big\{x\in U\,:\, \B_2^d(x,\varepsilon) \subseteq U \big\}\,.
\]
Since $U$ is assumed to be bounded there exists some $R\in(0,\infty)$ such that $U\subseteq  \B_2^d(0,R)$. Let $\eps>0$.  Then, for every $\omega\in \Omega_2$ there exists $N= N(\eps,\omega)\in\N$ such that for all $n>N$,  we have that
\begin{align}\label{eq:superset}
\frac{1}{\sqrt{n}}(GG^*)^{1/2}(U) \supseteq U^{\circ \varepsilon R}\,.
\end{align}
This means that, on $\Omega_1\cap \Omega_2$,
\begin{eqnarray*}
\liminf_{n\to\infty}\frac{1}{n} \log \mu^I_{\otimes n}(U)
& = & \liminf_{n\to\infty} \frac{1}{n} \log \mu^G_{\otimes n}\Big( \frac{1}{\sqrt{n}}(GG^*)^{1/2}(U)\Big) \cr
& \geq & \liminf_{n\to\infty} \frac{1}{n} \log \mu^G_{\otimes n}\big( U^{\circ \varepsilon R}\big) \cr
& \geq & -\inf_{x\in U^{\circ \varepsilon R}} \Lambda^*(x)\,,
\end{eqnarray*}
where we used \eqref{eq:21-06-1}, \eqref{eq:superset}, and Proposition \ref{prop:ldp gaussian setting} in this order.
Since the infimum of $\Lambda^*$ is attained either in (the interior of) $U$ or on the boundary $\overline U\setminus U$, we consider these two cases separately.

First, we assume that the infimum is attained at $x_{0} \in U$. Then there exists some $\varepsilon_0\in(0,\infty)$ such that $\B_2^d(x_0,\varepsilon_0)\subseteq U$\,. Letting $\varepsilon\downarrow 0$, we find that $x_0\in U^{\circ\varepsilon R}$ once $\varepsilon<\varepsilon_0/R$\,. Hence, for $\varepsilon\downarrow 0$, we obtain
\[
\liminf_{n\to\infty}\frac{1}{n} \log \mu^I_{\otimes n}(U) \geq -\Lambda^*(x_0) = -\inf_{x\in U}\Lambda^*(x)
\]
on $\Omega_1\cap\Omega_2$.
This shows the weak LDP lower bound for bounded open sets $U$ when the infimum of $\Lambda^*$ is attained in $U$.

Now, we assume that the infimum is attained at a point $\overline x$ on the boundary $\overline U\setminus U$. In this case, we find a sequence $(x_n)_{n\in\N}\subset U$ such that, as $n\to\infty$,
\[
\Lambda^*(x_n) \downarrow \Lambda^*(\overline x) = \inf_{x\in U}\Lambda^*(x)\,.
\]
More precisely, let $\delta\in(0,\infty)$. Then there exists $N\in\N$ such that
\[
\Lambda^*(x_N) \leq \Lambda^*(\overline x)+\delta = \inf_{x\in U}\Lambda^*(x) +\delta\,.
\]
But for this $N\in\N$ and corresponding $x_N$ there exits $\varepsilon_1\in(0,\infty)$ such that $x_N\in U^{\circ \varepsilon_1 R}$. Therefore,
\[
\inf_{x\in U}\Lambda^*(x) \leq \inf_{x\in U^{\circ \varepsilon_1 R}}\Lambda^*(x) \leq \Lambda^*(x_N) \leq \inf_{x\in U}\Lambda^*(x) +\delta\,.
\]
The latter holds in fact if $\eps_1$ is replaced by any $\varepsilon \leq \varepsilon_1$. Therefore, letting $\varepsilon\downarrow 0$ (and so eventually $\varepsilon\leq \varepsilon_1$) we obtain
\[
\liminf_{n\to\infty}\frac{1}{n} \log \mu^I_{\otimes n}(U) \geq -\inf_{x\in U^{\circ \varepsilon R}} \Lambda^*(x) \geq - \inf_{x\in U}\Lambda^*(x) - \delta
\]
on $\Omega_1\cap\Omega_2$.
Letting $\delta\downarrow 0$, we obtain the weak LDP lower bound for bounded open sets $U$ when the infimum of $\Lambda^*$ is attained on the boundary of $U$.

To conclude the almost sure weak LDP, it is now left to make the transition from bounded open sets to arbitrary open sets. This can be done since the rate function $\Lambda^*$ is good. Let $U\subseteq \R^d$ be any open set. Since $\Lambda^*$ is a good rate function it has compact level sets and so there exists $R\in(0,\infty)$ such that
\[
\inf_{x\in U}\Lambda^*(x) = \inf_{x\in U\cap (\B_2^d(0,R))^\circ}\Lambda^*(x)\,.
\]
Therefore, we obtain from the almost sure weak LDP lower bound for bounded open sets that, on $\Omega_1\cap\Omega_2$,
\begin{align*}
\liminf_{n\to\infty}\frac{1}{n} \log \mu^I_{\otimes n}(U) &\geq \liminf_{n\to\infty}\frac{1}{n} \log \mu^I_{\otimes n}\big(U\cap (\B_2^d(0,R))^\circ\big)\\
& \geq \inf_{x\in U\cap (\B_2^d(0,R))^\circ}\Lambda^*(x) =\inf_{x\in U}\Lambda^*(x)\,.
\end{align*}
This completes the proof of the almost sure weak LDP and it is left to prove almost sure exponential tightness, that is, exponential tightness on a subset of $\Omega$ with $\Pro$-measure $1$.

\vskip 2mm
\noindent\emph{Exponential tightness:} Let $C\in(0,\infty)$. By Proposition \ref{prop:ldp gaussian setting} the sequence $\mu_{\otimes n}^G$ satisfies (on a set of measure one) an LDP with speed $n$ and a good rate function. As a consequence, the sequence of measures is exponentially tight, i.e., there exists $R\in(0,\infty)$ such that
\begin{align}\label{eq:exponential tightness for G-setting}
\limsup_{n\to\infty} \frac{1}{n}\log \mu_{\otimes n}^G\big(\R^d\setminus \B_2^d(0,R/2)\big) \leq -C
\end{align}
on a set $\Omega_3\subset\Omega$ with $\Pro(\Omega_3)=1$.
As already used in \eqref{eq:inclusion}, it follows from \eqref{eq:almost sure convergence of scaled modulus to identity} that on $\Omega_2$ and for sufficiently large $n\in\N$, we have
\begin{align}\label{eq:inclusion for exponential tightness}
\frac{1}{\sqrt{n}}(GG^*)^{1/2}\big(\R^d\setminus \B_2^d(0,R)\big) \subseteq \R^d\setminus \B_2^d(0,R/2)\,.
\end{align}
Therefore, we obtain for large enough $n\in\N$ that
\begin{eqnarray*}
\mu_{\otimes n}^I\big(\R^d\setminus \B_2^d(0,R)\big) & = & \mu_{\otimes n}^G\Big(\frac{1}{\sqrt{n}}(GG^*)^{1/2}\big(\R^d\setminus\B_2^d(0,R)\big)\Big) \cr
& {\leq} & \mu_{\otimes n}^G\big(\R^d\setminus \B_2^d(0,R/2) \big)
\end{eqnarray*}
on $\Omega_2\cap\Omega_3$, were we used \eqref{eq:inclusion for exponential tightness} in the last step.
Taking logarithms and multiplying by $1/n$, an application of the limit superior as $n\to\infty$ together with the previous inclusion and \eqref{eq:exponential tightness for G-setting} shows that
\[
\limsup_{n\to\infty} \frac{1}{n}\log \mu_{\otimes n}^I\big(\R^d\setminus \B_2^d(0,R)\big) \leq \limsup_{n\to\infty} \frac{1}{n}\log \mu_{\otimes n}^G\big(\R^d\setminus \B_2^d(0,R/2)\big) \leq -C
\]
holds on $\Omega_2\cap\Omega_3$.
This completes the proof of the almost sure exponential tightness, which, together with what has previously been shown and Proposition \ref{prop:equivalence weak and full LDP}, implies the full LDP on the set $\Omega_1\cap\Omega_2\cap\Omega_3$ which satisfies $\Pro(\Omega_1\cap\Omega_2\cap\Omega_3)=1$. This eventually completes the proof of the almost sure LDP stated in Theorem \ref{thm:as LDP}.
\end{proof}

\section{Law of Large Numbers: Proof of Theorem \ref{thm:SLLN}}\label{sec:slln}

We shall now proceed with a proof for the strong law of large numbers stated in Theorem \ref{thm:SLLN}. As done before, we start in the Gaussian setting and then make a transition to the general case.

Consider a Gaussian random matrix $G=(g_{ij})_{i,j=1}^{d,n}:\R^n\to\R^d$ with independent standard Gaussian entries. If $e_1,\ldots,e_n$ is the standard orthonormal basis of $\R^n$, then  $X_1:= Ge_1,\ldots,X_n:=G e_n$ are independent standard Gaussian random vectors in $\R^d$.  We have that
\begin{align*}
{1\over n}G[-1,1]^n={1\over n}\bigoplus_{i=1}^nG[-e_i,e_i] = {1\over n}\bigoplus_{i=1}^n[-X_i,X_i],
\end{align*}
where $\oplus$ stands for the Minkowski sum. By the strong law of large numbers for random convex sets \cite{ArtsteinVitale}, we have that
\begin{align*}
d_H\Big({1\over n}G[-1,1]^n , E\Big)\toas 0
\end{align*}
on the space $\mathcal{K}(\R^d)$, where $E=\E[-X_1,X_1]$ is the convex set in $\R^d$ whose support function $h_E(u)$, $u\in\R^d$, is given by the identity
\begin{align*}
h_E(u) = \E h_{[-X_1,X_1]}(u) = \E \sup_{t\in[-X_1,X_1]}\langle u,t\rangle_2 = \E|\langle u,X_1\rangle_2|,\qquad u\in\R^d.
\end{align*}
Since $\langle u,X_1\rangle_2$ has distribution $\mathcal{N}(0,\|u\|_2^2)$, we conclude that
\begin{align*}
h_E(u) = \sqrt{2\over\pi}\,\|u\|_2^2,\qquad u\in\R^d.
\end{align*}
On the other hand, this is precisely the support function of a centered Euclidean ball in $\R^d$ with radius $\sqrt{2/\pi}$. In other words this means that there is a set $\Omega_1\subseteq\Omega$ with $\Pro(\Omega_1)=1$ with the following property on $\Omega_1$: for all $\varepsilon_1>0$ there exists $N_1\in\N$ such that $\dint_H(n^{-1}G[-1,1]^n,E)\leq\varepsilon_1$ whenever $n\geq N_1$.

To make the transition from $G$ to $I^*$, we use Lemma \ref{lem:uniform distribution Stiefel manifold} and write
\begin{align*}
{1\over\sqrt{n}}I^*([-1,1]^n) = {1\over\sqrt{n}}(GG^*)^{-1/2}G[-1,1]^n = \sqrt{n}(GG^*)^{-1/2}\,{1\over n}G[-1,1]^n.
\end{align*}
According to \eqref{eq:almost sure convergence of scaled modulus to identity} there exists $\Omega_2\subseteq\Omega$ with $\Pro(\Omega_2)=1$ such that on $\Omega_2$, we have that for all $\varepsilon_2>0$ there exists $N_2\in\N$ with the property that $\|\sqrt{n}(GG^*)^{-1/2}-\id_{d\times d}\|_{\rm op}\leq\varepsilon_2$ for all $n\geq N_2$. Using  \eqref{eq:inclusion} we conclude that, on $\Omega_1\cap\Omega_2$ we have that for each $\varepsilon>0$ there exists $N\in\N$ with the property that
\begin{align*}
{1\over\sqrt{n}}I^*([-1,1]^n) \subseteq \Big(\sqrt{2\over\pi}+\varepsilon\Big)\B_2^d + \varepsilon\Big(\sqrt{2\over\pi}+\varepsilon\Big)\B_2^d = (1+\varepsilon)\Big(\sqrt{2\over\pi}+\varepsilon\Big)\B_2^d
\end{align*}
for $n\geq N$.
Similarly, using \eqref{eq:superset}, we see that on $\Omega_1\cap\Omega_2$ for each $\varepsilon>0$ there exists $N\in\N$ with the property that
\begin{align*}
{1\over\sqrt{n}}I^*([-1,1]^n) \supseteq \Big(\Big(\sqrt{2\over\pi}-\varepsilon\Big)\B_2^d\Big)^{\circ \varepsilon\big(\sqrt{2\over\pi}-\varepsilon\big)} \supseteq (1-2\varepsilon)\Big(\sqrt{2\over\pi}-\varepsilon\Big)\B_2^d
\end{align*}
for $n\geq N$.
Letting now $\varepsilon\downarrow 0$ and noting that $\Pro(\Omega_1\cap\Omega_2)=1$ proves Theorem \ref{thm:SLLN}.\hfill $\Box$

\begin{proof}[Proof of Corollary \ref{cor:slln volume}]
The statement of the corollary follows directly from Theorem \ref{thm:SLLN} together with the continuous mapping theorem and the fact that the intrinsic volumes are continuous functionals on the space of compact convex sets endowed with the Hausdorff distance.
\end{proof}

\subsection*{Acknowledgement}
ZK has been supported by the German Research Foundation under Germany's Excellence Strategy EXC 2044 – 390685587, \textit{Mathematics M\"unster: Dynamics - Geometry - Structure}. JP has been supported by the Austrian Science Fund (FWF) Project P32405 ``Asymptotic Geometric Analysis and Applications'' and a visiting professorship from the University of Bochum and its Research School PLUS. 
ZK and CT have been supported by the DFG Scientific Network \textit{Cumulants, Concentration and Superconcentration}.

\bibliographystyle{plain}
\bibliography{LDmeasureprojection}

\end{document}